\documentclass[graybox]{svmult}

\usepackage{type1cm}        % activate if the above 3 fonts are
\usepackage{makeidx}         % allows index generation
\usepackage{graphicx}        % standard LaTeX graphics tool
\usepackage{multicol}        % used for the two-column index
\usepackage[bottom]{footmisc}% places footnotes at page bottom

\usepackage{newtxtext}       % 
\usepackage[varvw]{newtxmath}       % selects Times Roman as basic font
\usepackage{mathtools}
\usepackage[shortlabels]{enumitem}
\usepackage{todonotes}

\newcommand{\GammaInt}{\Gamma^{\text{i}}_h}
\newcommand{\GammaExt}{\Gamma^{\text{b}}_h}
\newcommand{\Mesh}{\mathcal{M}_h}
\newcommand{\trace}[1]{\operatorname{tr}_{\gamma}(#1)}
\newcommand{\avg}[1]{\left\lbrace\mskip-5mu\lbrace{#1}\right\rbrace\mskip-5mu\rbrace}
\newcommand{\dspace}{\mathcal{V}^r_h}
\newcommand{\binplus}{{} + }
\newcommand{\binminus}{{} - }
\newcommand{\phantomplus}{\: \phantom{+} \:}
\newcommand{\phantomminus}{\: \phantom{-} \:}
\newcommand{\phantomeq}{\mathrel{\phantom{=}}}
\newcommand{\cI}{\mathcal{I}}
\newcommand{\IE}{\mathbb{I}(E)}
\newcommand{\extop}[1]{\mathcal{L}_{#1}}
\newcommand{\iden}{\operatorname{id}}

\newcommand{\polynspace}{\mathcal{P}^r(\Omega)^m}
\newcommand{\globaltrace}{\operatorname{tr}}

\newcommand{\tlap}[1]{\vbox to 0pt{\vss\hbox{#1}}}

\makeindex             % used for the subject index
\begin{document}

\title*{On the consistency of the Domain of Dependence cut cell stabilization}
% Use \titlerunning{Short Title} for an abbreviated version of
% your contribution title if the original one is too long
\author{Gunnar Birke\orcidID{0000-0002-2008-6679},\\ Christian Engwer\orcidID{0000-0002-6041-8228},\\Jan Giesselmann\orcidID{0009-0008-0217-7244} and\\ Sandra May\orcidID{0000-0001-9178-6595}}
% Use \authorrunning{Short Title} for an abbreviated version of
% your contribution title if the original one is too long
\institute{Gunnar Birke \at Applied Mathematics Münster, Münster University, Germany, \email{gunnar.birke@uni-muenster.de}
\and Christian Engwer \at Applied Mathematics Münster, Münster University, Germany, \email{christian.engwer@uni-muenster.de}
\and Jan Giesselmann \at Department of Mathematics, Technical University of Darmstadt, Germany, \email{jan.giesselmann@tu-darmstadt.de}
\and Sandra May \at Institute of Mathematics, TU Berlin, Germany \email{may@math.tu-berlin.de}}
%
% Use the package "url.sty" to avoid
% problems with special characters
% used in your e-mail or web address
%
\maketitle

\abstract{So called cartesian cut cell meshes provide efficient ways to generate meshes but do require tailored numerical methods to not suffer from stabilization issues, especially in the hyperbolic regime where the application of explicit time stepping schemes is common. In this scenario, due to potentially arbitrarily small cut cells, an infeasible restriction is imposed on the time step size.\newline\indent
The \emph{Domain of Dependence} (DoD) stabilization allows for a time step size based on the underlying Cartesian mesh. Being an extension of a discontinuous Galerkin (DG) method,  one would expect similar accuracy properties as in the pure DG case. While numerical results do support this expectation, on the analytical level this has only been investigated thoroughly for $k=0$.\newline\indent
Error analysis typically hinges on a consistency result. In this contribution we prove such a result for the DoD stabilization given an arbitrary polynomial degree and an exact solution of sufficient regularity. This in turn could open the way towards a more refined analysis of the method even in the high-order case.}

\section{Introduction}

Meshing of complex geometries for numerical simulations poses a significant challenge. High-accuracy requirements on the requested mesh can impose serious costs when unstructured representations are used. Cartesian cut-cell meshes provide a similar level of accuracy by embedding objects of interest in a cartesian mesh and simply cutting them out, leading to a more efficient meshing procedure. However, this does not come for free, since the cutting process gives up most guaranties on mesh cell properties such as shape regularity and quasi-uniformity. Consequently, numerical methods on cut-cell meshes need to handle cells that are arbitrarily small or highly anisotropic in a special way. In the context of hyperbolic problems explicit time-stepping methods are often employed, where tiny cut-cells lead to prohibitively small time steps, rendering classical scheme infeasible. This is often called the \textit{small cell problem}.

There are essentially two approaches to overcome the \textit{small cell problem}. One option is to adjust the mesh and merge small cells with large neighbors. This is often called \textit{cell merging}/\textit{cell agglomeration} and is for example investigated in \cite{cell-merging-1, cell-merging-2, cell-merging-3, cell-merging-4}.

The alternative is to develop schemes that allow a time step size based on the original cartesian mesh even in the presence of small cut cells. Over the years a variety of such schemes have emerged, often based on a finite volume (FV) or discontinuous Galerkin (DG) scheme that are extended by additional stabilization terms. Examples are the flux redistribution method~\cite{flux-redist-2}, the \textit{h-box}~method~\cite{h-box-1, h-box-2}, dimensional-split schemes~\cite{dim-split-1, dim-split-2}, \textit{mixed~implicit~explicit} schemes~\cite{mixed-exp-imp}, an \textit{active~flux} method on cut cells~\cite{active-cut}, the \textit{state~redistribution} method~\cite{srd-1, srd-2}, the \textit{Domain of Dependence (DoD)} stabilization~\cite{dod-1, dod-2} and \textit{ghost-cell} methods in the hyperbolic context~\cite{ghost-cell-1, ghost-cell-2}.

Proving error estimates for these methods is delicate, one reason being the often very technical constructions due to missing geometric guaranties on the cut-cells. To our knowledge, in the hyperbolic case on cut-cell meshes such error analysis is mostly restricted to 1D,~e.g.~\cite{h-box-1, dim-split-2, mixed-exp-imp-2, ghost-cell-1, ghost-cell-2}. In \cite{dod-error-estimate} the authors were able to prove a two-dimensional result for the linear advection equation based on the DoD stabilization, although only for a first-order method. Unfortunately, the consistency result, underlying the final proof in \cite{dod-error-estimate} does not allow a straightforward generalization to high-order discretizations.

Our goal in this work is not to prove a complete error estimate but rather to focus on a generalized consistency result, extending to high-order schemes, under the assumption of sufficient regularity.
%An exploitation for the error analysis in the high-order case will be part of future research.

The article is organized as follows: We will first state the (linear) PDEs that we are considering in this work and introduce some mesh related terms and the unstabilized DG scheme. Then we will define {so called extension operators, one of the key building-blocks of the DoD stabilization,}
on the discrete space and derive their extension onto the analytical solution space. Finally we give consistency results for two different versions of the DoD stabilization.
\section{Preliminaries}

%\subsection{General PDE Setting}
We consider a general first-order linear (symmetric) hyperbolic system of conservation-form
\begin{subequations}
\label{eq:problem}
	\begin{alignat}{3}
		\label{eq: general pde}\partial_t u + A_1 \partial_x u + A_2 \partial_y u = \partial_t u + \nabla \cdot f(u) & = 0 & \quad & \text{in } \Omega \times (0,T),\\ %\Omega \subset \mathbb{R}^2,\\
		\label{eq: wall boundary condition} Eu & = 0 & & \text{on } \partial \Omega,\\
		u & = u_0 & & \text{on } \Omega \text{ at } t=0,%\Omega \times \{ t = 0 \},
	\end{alignat}
\end{subequations}
where $\Omega \subset \mathbb{R}^2$ is open and bounded, \tlap{$u = (u_1, \ldots, u_m)^T$}, the flux function $f$ is given as
\begin{align}
	\label{eq:physical-flux}
	f(u) &= \begin{pmatrix}f_1(u)\\f_2(u)\end{pmatrix} 
	= \begin{pmatrix}A_1 u\\ A_2 u\end{pmatrix}, \quad A_1, A_2 \in \mathbb{R}^{m\times m},
\end{align}
and $E = E(x)$ is a matrix-valued boundary field to enforce boundary conditions.
In this work we considered two specific examples
\begin{enumerate}%[wide]
	\item The linear advection equation where we pick
	\begin{equation} \label{eq: advection eq}
		m = 1, \; \begin{pmatrix}
			A_1\\A_2
		\end{pmatrix} = \beta \in \mathbb{R}^2, \; E = \beta \cdot n - |\beta \cdot n| .
	\end{equation}
	In this case $E$ enforces a zero-inflow condition on the inflow boundary.
    On the outflow boundary now condition is enforced.
	\item The linear acoustics equation where we pick
	\begin{equation} \label{eq: wave eq}
		m = 3, \; A_1 = {\footnotesize\left(\begin{array}{@{}c@{~\,}c@{~\,}c@{}}
				0 & c & 0\\
				c & 0 & 0\\
				0 & 0 & 0
			\end{array}\right)}, \;
		A_2 = {\footnotesize\left(\begin{array}{@{}c@{~\,}c@{~\,}c@{}}
				0 & 0 & c\\
				0 & 0 & 0\\
				c & 0 & 0
			\end{array}\right)}, \; E = {\footnotesize\left(\begin{array}{@{}c@{~\,}c@{~\,}c@{}}
			0 & n_1 & n_2 \\
			0 & 0 & 0\\
			0 & 0 & 0
			\end{array}\right)},
	\end{equation}
			where $c > 0$ denotes the speed of sound. Let~\tlap{$u = (p, v_1, v_2)^T$} be a state of this equation, $p$ being the pressure and $v$ the velocity. Then \tlap{$E u = (v \cdot n, 0, 0^T = 0$} enforces a reflecting wall boundary condition.
\end{enumerate}
\vspace*{-1em}

\subsection{Base discretization}
The final method is based on a classical DG discretization and is augmented with an additional stabilization.
We start with a rectangular domain $\widehat{\Omega} \supset \Omega$ and a so-called background mesh $\widehat{\mathcal{M}}_h$ of $\widehat{\Omega}$, a structured Cartesian mesh
of mesh size $h$,
such that $\bigcup_{\widehat{E} \in \widehat{\mathcal{M}}_h} \overline{\widehat{E}} = \overline{\widehat{\Omega}}$.
Based on this we construct a cut-cell mesh on $\Omega$ consisting of elements
\[
\mathcal{M}_h = \{ E = \widehat{E} \cap \Omega : \widehat{E} \in \widehat{\mathcal{M}}_h \},
\]
on which we define a discontinuous Galerkin space
\[
\dspace \coloneqq \{ v_h \in (L^2(\Omega))^m \; \big| \; \forall E \in \mathcal{M}_h : (v_h)_E \coloneqq (v_h)_{|E} \in P^r(E)^m \}.
\]
To formulate our numerical schemes on this space we will need additional notation. Let the set of internal and boundary faces be given by
\begin{align*}
	\GammaInt & \coloneqq \{ \gamma = \bar{E}_1 \cap \bar{E}_2 : |\gamma| > 0, E_1, E_2 \in \mathcal{M}_h \},\\
	\GammaExt & \coloneqq \{ \gamma = \bar{E} \cap \partial \Omega : E \in \mathcal{M}_h \}, \quad \Gamma_h \coloneqq \GammaInt \cup \GammaExt,
\end{align*}
where $|\gamma|$ denotes the 1D volume (length).
Given an element $E \in \Mesh$ we denote its sets of
internal, boundary, and all faces by
\begin{gather*}
	\GammaInt(E) \coloneqq \{ \gamma \in \GammaInt : \gamma \cap \partial E \neq \emptyset \}, \quad \GammaExt(E) \coloneqq \{ \gamma \in \GammaExt : \gamma \cap \partial E \neq \emptyset \}, \\%\quad
	\text{and} \quad \Gamma_h(E) \coloneqq \GammaInt(E) \cup \GammaExt(E).
\end{gather*}
We further define the index-set $\mathbb{I}(E) = \{1, \ldots, |\Gamma_h(E)| \}$, consecutively numbering all faces, such that any $i \in \mathbb{I}(E)$ identifies exactly one face $\gamma_i \in \Gamma_h(E)$. If $\gamma_i \in \GammaInt(E)$ is an internal face we denote by $E_i$ the corresponding neighbor element. If $\gamma_i \in \GammaExt(E)$ and there is no physical neighbor at this face, we will abuse notation by still writing $E_i$ and using this as an index variable later.

For each internal face $\gamma = \bar{E}_1 \cap \bar{E}_2 \in \GammaInt$ we fix a unique left-sided cell $E_{\gamma}^L \in \{E_1, E_2 \}$. The corresponding right-sided cell will be denoted by $E_{\gamma}^R$. 
The unit normal vector $n_{\gamma}$ on $\gamma$ is uniquely defined to point from left to right, i.e.,~$n_{\gamma} = n_{E_{\gamma}^L} |_{\gamma}$.

Given a face $\gamma = \bar{E}_{\gamma}^L \cap \bar{E}_{\gamma}^R \in \GammaInt$ and a function $w$ defined on $E_{\gamma}^L \cup E_{\gamma}^R$ such that both $w|_{E_{\gamma}^L}$ and $w|_{E_{\gamma}^R}$ have well-defined traces on $\gamma$, we define left and right values of $w$ as
\[
w_{\gamma}^L \coloneqq \trace{w|_{E_{\gamma}^L}}, \; w_{\gamma}^R \coloneqq \trace{w|_{E_{\gamma}^R}},
\]
and averages and jumps
\[
\avg{w}_{\gamma} \coloneqq \frac{1}{2} (w^L_{\gamma} + w^R_{\gamma}), \quad \llbracket w \rrbracket_{\gamma} \coloneqq w^L_{\gamma} - w^R_{\gamma}.
\]

In the DG discretization of the wave equation, the reflecting boundary condition $Eu = v \cdot n = 0$ in~\eqref{eq: wave eq} is implemented weakly by reflecting state arguments of the numerical flux. Furthermore, the DoD stabilization requires a reflection of polynomial functions. We introduce different operators to handle both, abusing notation by using $\mathfrak M_n$ to denote anyone of them.
\begin{definition}[Mirroring operator]	
	Let a cell $E \in \Mesh$ and a boundary face $\gamma \in \GammaExt(E)$ be given.
	For a state $u=(p,v)^T$ we define the state mirrored at $\gamma$ component wise as
	\begin{equation} \label{eq: reflecting wall boundary operator}
		\mathfrak{M}_n : \mathbb{R}^3 \to \mathbb{R}^3,
		\quad
		\begin{pmatrix}
			p \\ v
		\end{pmatrix} \mapsto 
		\begin{pmatrix}
			p \\ v - 2 (v \cdot n) n
		\end{pmatrix}.
	\end{equation}
	
	Furthermore we denote the orthogonal projection of $x \in E$ onto $\gamma$ as $x^\perp$ and
    %$n_\gamma^\perp$ 
    $n$ %=(x^\perp - x) / \|x^\perp - x\|$ 
    as the unit outer normal in $x^\perp$.
	With this we define the generalized mirroring operator
	\begin{equation}\label{eq: generalized reflecting wall boundary operator}
		\mathfrak{M}_n : \mathcal{P}^r(E)^3 \to \mathcal{P}^r(E)^3,
		\quad
		\bigg( x \mapsto
		\begin{pmatrix}
			p(x) \\ v(x)
		\end{pmatrix} \bigg) \mapsto 
		\bigg(x \mapsto
		\begin{pmatrix}
			p(x) \\ v(x) - 2 (v(x^\perp) \cdot n) n
		\end{pmatrix}\bigg),
	\end{equation}
 %   \JGcomment{Warum wird in (6) nicht $n_\gamma^\perp$ benutzt?\\CE: es wird später nie wieder genutzt, sollen wir nicht einfach direkt $n$ als Normale in $x^\perp$ definieren?\\GB: Ja\\
 %   JG: Aus meiner Sicht steht hier die Umsetzung aus. Oben wird ein $n_\gamma^\perp$ eingeführt und dann nicht genutzt.\\
  %  CE: Ich habe jetzt nur $n$ genutzt und oben als entsprechenden Normalenvektor eingeführt. Das präzisere $n_\gamma^\perp$ sprengt das Layout und ich würde sagen es jetzt konsistent und ausreichend klar.}
	which allows to evaluate $\mathfrak M_n(u)(x) ~\forall x\in E$. Note that this definition coincides on $\gamma$ with the definition of mirroring the state \eqref{eq: reflecting wall boundary operator}.
\end{definition}
Let $u^\mu, u^\nu \in \mathbb{R}^m$ be two states. We denote by $f^S_n(u^\mu, u^\nu) \coloneqq \frac{1}{2} [f_n(u^\mu) + f_n(u^\nu)]$
the central flux and by $S_n(u^\mu, u^\nu)$ a dissipative flux that is consistent, i. e. $S_n(u, u) = 0$. For the linear advection equation~\eqref{eq: advection eq} we will always pick $S_n(u^\mu, u^\nu) = \frac{|\beta \cdot n|}{2}(u^\mu - u^\nu)$, leading to an upwind discretization. For the linear wave equation~\eqref{eq: wave eq} we can, for example, pick $S_n(u^\mu, u^\nu) = \frac{c}{2}(u^\mu - u^\nu)$.

Our base scheme reads: Find $u_h(t) \in \dspace$ such that
\begin{equation}\label{eq: base scheme}
	\int_{L^2(\Omega)} \langle \partial_t u_h(t), w_h \rangle  + a_h(u_h(t), w_h) = 0 \quad \forall w_h \in \dspace,
\end{equation}
where for the advection equation~\eqref{eq: advection eq}, with 
$\langle \cdot, \cdot \rangle$ the euclidian innre product for state vectors,
\begin{align} \label{eq:base-bilinear-form advection eq}
	\begin{split}
		a_h(u_h, w_h) = & - \sum_{E \in \mathcal{M}_h} \int_E f(u_h) \cdot \nabla w_h + \sum_{\gamma \in \GammaInt} \int_{\gamma} [f_n^S + S_n]  (u_{\gamma}^L, u_{\gamma}^R) \llbracket w_h \rrbracket\\
		& + \sum_{\gamma \in \GammaExt} \int_{\gamma} (\beta \cdot n)^+ u_h w_h,
	\end{split}
\end{align}
with $a^+ = \frac{1}{2} (|a|- a)$ beeing the positive part of a real number,
and for the wave equation~\eqref{eq: wave eq}
\begin{align} \label{eq:base-bilinear-form wave eq}
\begin{split}
	a_h(u_h, w_h) = & - \sum_{E \in \mathcal{M}_h} \int_E f(u_h) \cdot \nabla w_h + \sum_{\gamma \in \GammaInt} \int_{\gamma} \langle [f_n^S + S_n]  (u_{\gamma}^L, u_{\gamma}^R), \llbracket w_h \rrbracket \rangle\\
	& + \sum_{\gamma \in \GammaExt} \int_{\gamma} \langle [f_n^S + S_n] (u_h, \mathfrak{M}_n(u_h)) , w_h \rangle.
\end{split}
\end{align}
Consistency of the base scheme is discussed, for example, in~\cite{error-estimate-2} and~\cite{error-estimate-1}.
\section{Extension operators}
One of the primary tools of the DoD method, introduced in \cite{dod-1}, is the concept of extension operators:
\begin{definition}[Extension operator] 
	Given a discrete function $w_h \in \dspace$ and a cell ${E \in \mathcal{M}_h}$ we introduce the extension operator
	\[
	\extop{E}: \dspace \to (P^r(\Omega))^m, \quad w_h \mapsto \extop{E}(w_h)
	\]
	defined via $\extop{E}(w_h)(x) = w_h|_{E}(x), \; x \in E$, i.e., we select the polynomial components of $w_h$ on $E$ and extend them to the complete domain. This definition is well-posed since defining a polynomial function on an open domain is equivalent to defining it on the complete real space.
\end{definition}
\begin{definition}[Reflected extension operator]
	Let $w_h \in \dspace$ be a discrete function, $E \in \mathcal{M}_h$ an element and $\gamma \in \GammaExt$ a boundary face on which a reflecting wall boundary condition is applied.
	We define the extension operator over this face $\gamma$ as
\begin{equation}
	\label{eq:mirroring-extension-op}
	\mathcal{L}_{E}^{\gamma}: \dspace \to P^r(\Omega)^3,
	\quad w_h \mapsto \mathcal{L}_{E}^{\gamma}(w_h) = \mathfrak{M}_{n_{\gamma}}(\extop{E}(w_h)).
\end{equation}
\end{definition}
From now on we will assume that the solution $u$ of equation~\eqref{eq:problem} lies in the space $V = H^{r+1}(\Omega)$ where $r$ is the polynomial degree of our discrete space $\dspace$. This regularity assumption will be crucial for our result. We further introduce the space
\begin{equation}
	V^{\mathfrak{M}} = \{ u \in V^3 : \: \globaltrace u - \mathfrak{M}_n(\globaltrace u) = 0 \text{ on } \partial \Omega \}
\end{equation}
%\JGcomment{Formal ist $\mathfrak{M}_{n_{\gamma}}$ bisher nur als Operator auf Räumen von Polynomen definiert.\\GB: Hab einen kurzen Satz eingeführt\\
%JG: Das löst mein Problem nicht.  Weder $u$ noch $\globaltrace u$ ist ein Polynom...\\
%CE: In (5) ist $\mathfrak{M}_{n_{\gamma}}$ zunächst als Abbildung auf dem State eingeführt worden.\\
%JG: Okay, dann verstehe ich es.\\
%CE: Aber man sollte es klarer herausstellen, ich habe jetzt nochmal explizit auf (5) verwiesen.}
which is the solution space to equation~\eqref{eq: wave eq}. {The above application of the mirroring operator has to be understood pointwise, i. e. $\mathfrak{M}_n(\globaltrace u)(x) = \mathfrak{M}_n(\globaltrace u(x))$, see \eqref{eq: reflecting wall boundary operator}, and the equality in the sense of traces.}

An important aspect of the proofs given in \cite{error-estimate-1} and \cite{error-estimate-2} is that the considered bilinear forms accept values in the space $V + \dspace$ in the first argument. This means that we 
% will have to extend the extension operators defined above to this space 
have to generalize the definition of the extension operator to this space
(or $V^{\mathfrak{M}} + \dspace$ in case of equation~\eqref{eq: wave eq}). This 
is %will be 
achieved by %constructing 
mappings of the form
\begin{subequations}
	\begin{align}
			\text{id} + \extop{E} & : V + \dspace \to V + \polynspace,\\
			\text{id} + \extop{E}^\gamma & : V^{\mathfrak{M}} + \dspace \to V^{\mathfrak{M}}  + \polynspace
	\end{align}
\end{subequations}
 where $\iden$ is the identity mapping on either $V$ or $ V^{\mathfrak{M}}$. 
Existence and uniqueness of such mappings is given by the {gluing Lemma of linear maps} which we state here in the context of general function spaces:
%The existence (and uniqueness of such mappings) is related to the intersection of the domains:
%\textcolor{orange}{1. Hier fehlt doch noch ein wenig, ich würde von der Notation irgendie noch den Raum $X := \operatorname{span}(U, W)$ benennen und evntl. direkt schon $U, W$ als subsapces benennen. 2. braucht man m.M.n. nicht unbedingt ein Zitat, da es ja eigentlich klassische Algebra ist, oder?! ... wir sollten es umformulieren und den Beweis weglassen.}
\begin{lemma} \label{lem: sum mapping}
	Let $X$ and $T$ be vector spaces, $U, W \subset X$ subspaces, $f: U \to T$ and $g: W \to T$ be linear maps. There is a unique linear map
	\begin{displaymath}
		f + g: U + W \to T
	\end{displaymath}
	such that $(f + g)_{|U} = f$ and $(f + g)_{|W} = g$ if and only if $f_{|U \cap W} = g_{|U \cap W}$, that is the two maps $f$ and $g$ agree on the intersection of $U$ and $W$.
\end{lemma}
%\begin{proof}
%\end{proof}
The other ingredient that we need is a connection between weak and strong differentiability of functions in our DG space. This statement can, for example, be found in~\cite{braess}:
\begin{lemma} \label{lem: braess sobolev}
	Let $k \geq 1$. A function $f  :\bar{\Omega} \to \mathbb{R}$ which is piecewise $C^\infty$ is in $H^k(\Omega)$ if and only if it is in $C^{k-1}(\bar{\Omega})$.
\end{lemma}

We can now show the following result.
\begin{lemma} \label{lem: intersection of spaces}
	Let $r \geq 0$, then it holds
	\begin{equation}
		V \cap \dspace = \polynspace.
	\end{equation}
\end{lemma}
\begin{proof}
	Let $u \in V \cap \dspace$. Since $u \in \dspace$ we know that $u$ is defined on $\bar{\Omega}$ and is piecewise $C^\infty$. Since $u \in V = H^{r+1} (\Omega)$ we deduce from Lemma~\ref{lem: braess sobolev} that $u \in C^r(\Omega)$. Since $u$ is a piecewise polynomial of degree $r$ and has $r$ continuous derivatives, we have $u \in \polynspace$.
\end{proof}

\begin{corollary} \label{cor: extop on sum space}
Let $E \in \Mesh$ be a mesh cell. Then the mapping
\begin{equation}
\iden + \extop{E} : V + \dspace \to V + \polynspace
\end{equation}
is well-defined. If $\gamma \in \GammaExt$ is a boundary face and $u$ is a solution to equation~\eqref{eq: wave eq}, the mapping
\begin{equation}
\iden + \extop{E}^\gamma : V^{\mathfrak{M}} + \dspace \to V^\mathfrak{M} + \polynspace
\end{equation}
is well-defined.
\end{corollary}
\begin{proof}
By Lemma~\ref{lem: intersection of spaces} and the fact that $\extop{E}$ is the identity on $\polynspace$ the first statement holds true by Lemma~\ref{lem: sum mapping}. For the second statement we first note that by Lemma~\ref{lem: intersection of spaces}
\begin{displaymath}
	V^\mathfrak{M} \cap \dspace = \{ u \in  \polynspace : u - \mathfrak{M}_n(u) = 0 \text{ on } \partial \Omega \}.
\end{displaymath}
Again we have that $\extop{E}^\gamma$ is the identity on $V^\mathfrak{M} \cap \dspace$ and we can apply Lemma~\ref{lem: sum mapping} again to get the second statement.
\end{proof}

\section{Consistency of the Domain of Dependence stabilization}

The DoD stabilization adds penalty terms on small cut cells to the base scheme~\eqref{eq: base scheme}. The general form is: Find $u_h(t) \in \dspace$ such that for all $w_h \in \dspace$
\begin{equation}
	\label{eq: stabilized scheme}
	\int_{\Omega} \langle \partial_t u_h(t), w_h \rangle  + a_h(u_h(t), w_h) + J_h(u_h(t), w_h) = 0.
\end{equation}
Small cells that require stabilization will be collected in a subset $\cI \subset \Mesh$. The (bilinear) form $J_h(u_h(t), w_h) = \sum_{E \in \cI} J^E_h(u_h(t), w_h)$ is defined cell-wise. Concrete cell-stabilization terms will be introduced below. Assuming an already consistent base scheme it is enough to show that
\begin{equation}
	J^E_h(u(t), w_h) = 0 \quad \forall w_h \in \dspace
\end{equation}
to prove consistency of the stabilized scheme.

\subsection{Linear advection equation}

To illustrate the application of Corollary~\ref{cor: extop on sum space} we first consider the DoD stabilization that was given in~\cite{dod-icosahom} for the linear advection equation \eqref{eq: advection eq}. This formulation assumes that for any small cut cell $E$ there is exactly one inflow face $E_{in}$ and is given as:

\begin{definition}
Let $E \in \cI$ a small cut cell, $E_{\text{in}}$ its only inflow face, $\eta_E \in [0, 1]$.
The cell stabilization term for equation~\eqref{eq: advection eq} is $J^E_h(u_h, w_h) = [J^{0, E}_h + J^{1, E}_h](u_h, w_h)$ with
\begin{subequations}
	\begin{align}
		J^{0, E}_h(u_h, w_h) & = \eta_E \int_{\partial E}  ( \beta \cdot n)^+ (\extop{E_{in}}(u_h) - \extop{E}(u_h)) \llbracket w_h \rrbracket,\\
		J^{1, E}_h(u_h, w_h) & = \eta_E \int_{ E}   (\extop{E_{in}}(u_h) - u_h)  \beta \cdot \nabla (\extop{E_{in}}(w_h) - w_h).
	\end{align}
\end{subequations}
\end{definition}

\begin{proposition}
	Let $u$ be the exact solution to equation~\eqref{eq: advection eq}. Then
	\[
	 J_h^{E} (u, w_h) = 0 \quad \forall E \in \cI \text{ and } \forall w_h \in \dspace.
	\]
\end{proposition}

\begin{proof}
By Corollary~\ref{cor: extop on sum space} we know that $\extop{E_{in}}(u) = \extop{E}(u) = u$. Thus we have
\[
J^{0, E}_h(u, w_h) = \eta_E \int_{\partial E}  ( \beta \cdot n)^+ (u- u) \llbracket w_h \rrbracket = 0 \quad \forall w_h \in \dspace,
\]
and similarly for $J^{1, E}_h(u, w_h)$.
\end{proof}

\subsection{Linear wave equation}

The proof of consistency for the DoD stabilization for the wave equation, c.f.~\cite{dod-wave-eq}, will follow along the same lines. The formulation of the stabilization terms is more involved and we will recall the necessary definitions.

\begin{definition}[Propagation forms] \label{def: propagation forms}
	Let $E \in \cI$ be a small cell {and 
    $\mathbb{I}(E)=\{1,\ldots,K\}$ the set of face indices of $E$}. We introduce a set of surface and volume propagation forms% as functionals
	\begin{align*}
		p_{ij}^E(u^{\mu}, u^{\nu}, w) & : (P^r(E))^3 \times (P^r(E))^3 \times (P^r(E))^3 \to \mathbb{R}, \; i, j \in \IE, i \neq j,\\
		p_V^E(u^{\mu}, u^{\nu}, w) & : (P^r(E))^3 \times (P^r(E))^3 \times (P^r(E))^3 \to \mathbb{R},\\
		p_V^{E, *}(u^{\mu}, u^{\nu}, w) & : (P^r(E))^3 \times (P^r(E))^3 \times (P^r(E))^3 \to \mathbb{R}.
	\end{align*}
  These forms are required to be symmetric in the first two arguments and linear in the third. They must further satisfy for all vector-valued polynomial functions $u^{\mu}, u^{\nu}, w \in \mathcal{P}^r(E)^3$ the following properties:
	\begin{enumerate}[i),labelindent=0pt,listparindent=0pt,align=left,leftmargin=0pt,labelwidth=0pt,itemindent=!]
		\begin{subequations}
			\item Let $i, j \in \mathbb{I}(E)$. Then it holds
			\begin{equation} \label{eq: propagation balance}
				p_{ij}^E(u^{\mu}, u^{\nu}, w) + p_{ji}^E(u^{\mu}, u^{\nu}, w) = p_V^E(u^{\mu}, u^{\nu}, w) + p_V^{E, *}(u^{\mu}, u^{\nu}, w).
			\end{equation}
			\item \label{lem: propagation forms consistency} Let $j \in \mathbb{I}(E)$. Then it holds
			\begin{equation} \label{eq: propagation forms consistency}
				\sum_{i \in \mathbb{I}(E), i \neq j} p_{ij}^E(u^{\mu}, u^{\nu}, w) = \int_{\gamma_j} \tfrac{1}{2} \langle f_n(u^{\mu}) + f_n(u^{\nu}), w \rangle.
			\end{equation}
		\end{subequations}
	\end{enumerate}
\end{definition}

%Compared to~\cite{dod-wave-eq} we will also need to impose the following property on our propagation forms:
In addition, the specific choice of propagation forms in~\cite{dod-wave-eq} satisfies the property
\begin{equation} \label{eq: propagation forms volume consistency}
	p_V^E(u, u, w) = \tfrac{2}{K(K-1)} \int_E f(u) \cdot \nabla w,
\end{equation}
which will be necessary for consistency.
%The examples of propagation forms given in~\cite{dod-wave-eq} satisfy this property.
%\textcolor{red}{Diese Eigenschaft ist nicht für die Energiestabilität wichtig, hier aber schon. Ist mir jetzt erst aufgefallen. Man kann das hier auch weiter abstrahieren, aber ist jetzt zuviel Aufwand.}
%\textcolor{orange}{ist das von der konkreten Wahl in \cite{dod-wave-eq} erfüllt? Einmal die beiden Aussagen umdrehen}

To formulate the cell stabilization we use the following notation: Let $E \in \cI$ be a small cut cell with $\gamma_i \in \Gamma_h(E)$ and $\gamma_j \in \Gamma_h(E)$ being two of its faces, $i, j \in \IE$. Then we denote by
\begin{equation} \label{eq: unified extension operator}
	\extop{\mathcal{E}}^{ij} =
	\begin{cases}
		\extop{\mathcal{E}}
		&\quad\text{if }
		\mathcal{E} = E \text{ or } \mathcal{E} = E_k, k \in \{i, j\} \text{ and } \gamma_k \in \GammaInt\\
		\extop{E_j}^{\gamma_i}
		&\quad\text{if } \mathcal{E} = E_i \text{ and }
		\gamma_i \not \in \GammaInt\\
		\extop{E_i}^{\gamma_j}
		&\quad\text{if } \mathcal{E} = E_j \text{ and }
		\gamma_j \not \in \GammaInt
	\end{cases}
\end{equation}
the extension operator between $\gamma_i$ and $\gamma_j$ that selects either of the previously defined extension operators based on whether we are extending from an internal or external boundary. We abuse notation here by writing $\mathcal{E} = E_i$ even if $\gamma_i \not \in \GammaInt$ for $i \in \IE$. 

\begin{definition}[Cell stabilization terms] {\label{def: cell stabilization terms}}
	Let $E \in \cI$ be a small cut cell. The cell-wise stabilization term $J^E_h = J_{h}^{0,E} + J_{h}^{1,E} + J_{h}^{s,E}$ is constructed from pairwise (for each pair of neighbors $i, j \in \IE, i \neq j$) contributions as follows
	\begin{subequations} \label{eq: dod cell terms}
		\begin{align}
			\label{eq: dod cell term a0}
			\begin{split}
				J^{0, E}_h(u_h, w_h) & = \phantomminus \eta_E \sum_{(i, j) \in \IE^2, \, i < j} J^{0, E}_{h, ij} (u_h, w_h)\\
				& \phantomeq \binminus \eta_E \sum_{\gamma \in \GammaInt(E)} \int_{\gamma} \langle \tfrac{1}{2}[f_n(u^L_{\gamma}) + f_n(u^R_{\gamma})], \llbracket w_h \rrbracket \rangle\\
				& \phantomeq \binminus \eta_E \sum_{\gamma \in \GammaExt(E)} \int_{\gamma} \langle \tfrac{1}{2}[f_n(u_h) + f_n(\mathfrak{M}_n(u_h))], w_h \rangle,
			\end{split}\\
			\label{eq: dod cell term a1}
			J^{1, E}_h(u_h, w_h) & = \phantomminus \eta_E \sum_{(i, j) \in \IE^2, \, i < j} J^{1, E}_{h, ij}(u_h, w_h),\\
			\label{eq: dod cell term s}
			\begin{split}
				J^{s, E}_h(u_h, w_h) & = \phantomminus \eta_E \sum_{(i, j) \in \IE^2, \, i < j} J^{s, E}_{ij}(u_h, w_h)\\
				& \phantomeq \binminus \eta_E \sum_{\gamma \in \GammaInt(E)} \int_{\gamma} \langle S_n(u^L_{\gamma}, u^R_{\gamma}), \llbracket w_h \rrbracket \rangle\\
				& \phantomeq \binminus \eta_E \sum_{\gamma \in \GammaExt(E)} \int_{\gamma} \langle S_n(u_h, \mathfrak{M}_n(u_h)), w_h \rangle.
			\end{split}
		\end{align}
	\end{subequations}
	For the surface contributions $J_{h,ij}^{0,E}$ of $J_{h}^{0,E}$
	we distinguish two cases: (1) two inner faces, i.e. $\gamma_i, \gamma_j \in \GammaInt$, and (2) one face being a reflecting boundary face.
	
	\begin{enumerate}[wide,label=(\arabic*)]
		\item For two inner faces $\gamma_i, \gamma_j \in \GammaInt$, associated with cut-cell neighbors $E_i$ and $E_j$, the surface stabilization term between $E_i$ and $E_j$ is given by
		\begin{equation}
			\begin{split}
				J_{h,ij}^{0,E}(u_h, w_h) & = \phantomplus p_{ij}^E(\extop{E_i}(u_h), \extop{E_j}(u_h), \extop{E}(w_h) - \extop{E_j}(w_h))\\
				& \phantomeq \binplus p_{ji}^E(\extop{E_i}(u_h), \extop{E_j}(u_h), \extop{E}(w_h) - \extop{E_i}(w_h)). 
			\end{split}
		\end{equation}
		
		\item If one of the two faces is a reflecting boundary face, the stabilization needs to be defined separately. W.l.o.g we assume that $\gamma_i \in \GammaExt$ and $\gamma_j \in \GammaInt$. The reflected surface stabilization term between $\gamma_i$ and $\gamma_j$ is then given by
		\begin{equation}
			\begin{split}
				J_{h,ij}^{0,E}(u_h, w_h) & = \phantomplus p_{ij}^E(\extop{E_j}^{\gamma_i}(u_h), \extop{E_j}(u_h), \extop{E}(w_h) - \extop{E_j}(w_h))\\
				& \phantomeq \binplus p_{ji}^E(\extop{E_j}^{\gamma_i}(u_h), \extop{E_j}(u_h), \extop{E}(w_h)). 
			\end{split}
		\end{equation}
	\end{enumerate}
	To define the contribution $J_{h,ij}^{1,E}$ we introduce a weighting factor $\omega_{\mathcal{E}} := \{ -1 \text{ if } \mathcal{E} = E; \frac{1}{2} \text{ else}\}$. With this the volume stabilization is given as
	\begin{equation} \label{eq: dod volume term ij}
		\begin{split}
			J_{h,ij}^{1,E}(u_h, w_h) & = \phantomplus \sum_{\mathcal{E} \in \{E, E_i, E_j\}} \omega_{\mathcal{E}} \big[p_V^E (\extop{E_i}^{ij}(u_h), \extop{E_j}^{ij}(u_h), \extop{\mathcal{E}}^{ij}(w_h))\\
			& \phantomeq \phantomplus \hphantom{\sum_{\mathcal{E} \in \{E, E_i, E_j\}} \omega_{\mathcal{E}} \big[]}- \tfrac{2}{K(K-1)} \int_{E}  f (\extop{\mathcal{E}}^{ij}(u_h)) \cdot \nabla \extop{\mathcal{E}}^{ij}(w_h) \big]\\
			& \phantomeq \binplus \: \sum_{\mathcal{E} \in \{E, E_i, E_j\}} \omega_{\mathcal{E}} p_V^{E,*} (\extop{E_i}^{ij}(w_h), \extop{E_j}^{ij}(w_h), \extop{\mathcal{E}}^{ij}(u_h)).
		\end{split}
	\end{equation}
Finally the contribution $J^{s, E}_{ij}$ to the dissipative stabilization is
\begin{equation} \label{eq: dod neighbor dissipation}
		\begin{split}
			J^{s, E}_{ij}(u_h, w_h) & = \phantomplus \tfrac{1}{6} \sum_{\gamma \in \Gamma_h(E) } \int_{\gamma} \langle S_n(\extop{E_i}^{ij}(u_h),\extop{E_j}^{ij}(u_h)), \extop{E_i}^{ij}(w_h) \mathord{-} \extop{E_j}^{ij}(w_h)) \rangle\\
			& \phantomeq \binplus \tfrac{1}{6} \sum_{\gamma \in \Gamma_h(E)} \int_{\gamma} \langle S_n(\extop{E_j}^{ij}(u_h),\extop{E_i}^{ij}(u_h)), \extop{E_j}^{ij}(w_h) \mathord{-}\extop{E_i}^{ij}(w_h)) \rangle.
		\end{split}
	\end{equation}
	
\end{definition}

\begin{proposition}
	Let $u$ be the exact solution to equation~\eqref{eq: wave eq}. Then
\[
J_h^{E} (u, w_h) = [J^{0, E}_h + J^{1, E}_h + J^{s, E}_h](u, w_h)= 0 \quad \forall E \in \cI \text{ and } \forall w_h \in \dspace.
\]
\end{proposition}
\begin{proof}
For $u$ being the exact solution, Corollary\,\ref{cor: extop on sum space} implies that
\begin{equation} \label{eq: exact solution ext op}
\extop{E_{i}}(u) = \extop{E}(u) = \extop{E}^\mathfrak{M}(u) =  \extop{E_i}^\mathfrak{M}(u)  = u.
\end{equation}
% Furthermore we know (or could imply from the above equations) that $\mathfrak{M}_n(u) = u$ at the boundary and also that $u^L_\gamma = u^R_\gamma$. We can thus rewrite the term~\eqref{eq: dod cell term a0} as
Furthermore $\mathfrak{M}_n(u) = u$ at the boundary and $u  ^L_\gamma = u^R_\gamma$. With this \eqref{eq: dod cell term a0} can be rewritten as
\begin{align*}
	\tfrac{1}{\eta_E} J^{0, E}_h(u, w_h) & = \phantomminus  \sum_{(i, j) \in \IE^2, i < j} J^{0, E}_{h, ij} (u, w_h) - \sum_{j \in \IE, \gamma_j \in \GammaInt(E)} \int_{\gamma_j} \langle f_n(u), \llbracket w_h \rrbracket \rangle\\
	& \phantomeq \binminus  \sum_{j \in \IE, \gamma_j \in \GammaExt(E)} \int_{\gamma_j} \langle f_n(u), w_h \rangle.
\end{align*}
We need to rewrite the first sum so that it matches the second and third. 
Using $p^E_{ij}(u,u,\cdot) = p^E_{ji}(u,u,\cdot)$
we rearrange terms inside the first sum to get
\begin{align*}
	\sum_{i, j \in \IE, i < j} & J^{0, E}_{h, ij} (u, w_h)\\
	& = \phantomplus \sum_{j \in \IE, \gamma_j \in \GammaInt(E)} \sum_{i \in \IE, i \neq j} p^E_{ij} (u, u, \extop{E}(w_h) - \extop{E_j}(w_h))\\
	& \phantomeq \binplus  \sum_{j \in \IE, \gamma_j \in \GammaExt(E)} \sum_{i \in \IE, i \neq j} p^E_{ij}(u, u, \extop{E}(w_h))\\
	& \overset{\mathclap{\eqref{eq: propagation forms consistency}}}{=} \phantomplus  \sum_{j \in \IE, \gamma_j \in \GammaInt(E)} \int_{\gamma_j} \langle f_n(u), \llbracket w_h \rrbracket \rangle + \sum_{j \in \IE, \gamma_j \in \GammaExt(E)} \int_{\gamma_j} \langle f_n(u), w_h \rangle.
\end{align*}
Thus $J^{0, E}_h(u, w_h) = 0$. Let us now consider the volume contributions~\eqref{eq: dod volume term ij} that make up the complete volume term~\eqref{eq: dod cell term a1}. Using~\eqref{eq: exact solution ext op} we have
	\begin{displaymath} 
	\begin{split}
		J_{h,ij}^{1,E}(u, w_h) & = \phantomplus \sum_{\mathcal{E} \in \{E, E_i, E_j\}} \omega_{\mathcal{E}} \big[p_V^E (u, u, \extop{\mathcal{E}}^{ij}(w_h))- \tfrac{2}{K(K-1)} \int_{E}  f (u) \cdot \nabla \extop{\mathcal{E}}^{ij}(w_h) \big]\\
		& \phantomeq \binplus \: \sum_{\mathcal{E} \in \{E, E_i, E_j\}} \omega_{\mathcal{E}} p_V^{E,*} (\extop{E_i}^{ij}(w_h), \extop{E_j}^{ij}(w_h), u)\\
		& \overset{\mathclap{\eqref{eq: propagation forms volume consistency}}}{=} \phantomplus \sum_{\mathcal{E} \in \{E, E_i, E_j\}} \omega_{\mathcal{E}} p_V^{E,*} (\extop{E_i}^{ij}(w_h), \extop{E_j}^{ij}(w_h), u) = 0
	\end{split}
\end{displaymath}
The last equation follows from the definition of the weighting factors $\omega_{\mathcal{E}}$ and the linearity of $p^{E, *}_V$ in the last argument. Now we can conclude $J_{h}^{1,E}(u, w_h) = 0$. Lastly we infer from~\eqref{eq: exact solution ext op} and the consistency of $S_n$ that $J^{s, E}_{ij}(u, w_h) = 0$ and $J^{s, E}(u, w_h) = 0$ by a term-wise inspection.
\end{proof}

\begin{acknowledgement}
The research of JG, CE and GB was supported by the Deutsche Forschungsgemeinschaft (DFG, German Research Foundation) - SPP 2410 Hyperbolic Balance Laws in
Fluid Mechanics: Complexity, Scales, Randomness (CoScaRa), specifically JG within the project 525877563 (A posteriori error estimators for statistical solutions
of barotropic Navier-Stokes equations) and CE and GB within project 526031774 (EsCUT: Entropy-stable high-order CUT-cell discontinuous Galerkin methods).
SM, CE and GB gratefully acknowledges support by the Deutsche Forschungsgemeinschaft (DFG, German Research Foundation) - 439956613 (Hypercut).
CE and GB further acknowledge support by the Deutsche Forschungsgemeinschaft (DFG, German Research Foundation) under Germany's Excellence Strategy EXC 2044 –390685587, Mathematics Münster: Dynamics–Geometry–Structure.
JG also acknowledges support by the German Science Foundation (DFG) via
grant TRR 154 (Mathematical modelling, simulation and optimization using the example of gas
networks), sub-project C05 (Project 239904186).
SM gratefully acknowledges support by the Deutsche Forschungsgemeinschaft (DFG, German Research Foundation) under Germany´s Excellence Strategy – The Berlin Mathematics Research Center MATH+ (EXC-2046/2, project ID: 390685689).
\end{acknowledgement}
\ethics{Competing Interests}{The authors have no conflicts of interest to declare that are relevant to the content of this chapter.}
%\textcolor{red}{Gibt es hier irgendwelche relevanten Konflikte?}
%%%%%%%%%%%%%%%%%%%%%%%% referenc.tex %%%%%%%%%%%%%%%%%%%%%%%%%%%%%%
% sample references
% %
% Use this file as a template for your own input.
%
%%%%%%%%%%%%%%%%%%%%%%%% Springer-Verlag %%%%%%%%%%%%%%%%%%%%%%%%%%
%
% BibTeX users please use
% \bibliographystyle{}
% \bibliography{}
%

\end{document}